\documentclass[11pt, a4paper, twoside,reqno]{amsart}
\usepackage{amsmath,amsthm,amsfonts,amssymb}
\usepackage[utf8]{inputenc}
\usepackage{mathrsfs}
\usepackage{graphics,graphicx}
\usepackage{ulem}
\usepackage{mathabx}
\usepackage[dvipsnames]{xcolor}
\usepackage{enumerate}
\usepackage{icomma}
\usepackage{empheq}
\usepackage{pdflscape}
\usepackage{epstopdf}
\usepackage{scalerel}
\usepackage{rotating}
\usepackage[shortlabels]{enumitem}
\usepackage[bookmarks=true,breaklinks=true,bookmarksnumbered = true,colorlinks=false,hidelinks]{hyperref}

\definecolor{arrow}{RGB}{154, 32, 64}

\theoremstyle{plain}
\newtheorem*{thm0}{Theorem}

\newtheorem{thm}{Theorem}[section]

\newtheorem{lem}[thm]{Lemma}  
\newtheorem{prop}[thm]{Proposition} 
\newtheorem{cor}[thm]{Corollary} 

\theoremstyle{definition}
\newtheorem{defn}[thm]{Definition} 
\newtheorem{ex}[thm]{Example}

\theoremstyle{remark}
\newtheorem{rmq}[thm]{Remark} 

\newtheorem{claim}[thm]{Claim}

\begin{document}
\normalem
\title{Homotopy braid groups are torsion-free}
\author{Emmanuel Graff}
	

 \begin{abstract}

We show that, for any number of components, the group of braids up to link-homotopy is torsion-free. This generalizes a result of Humphries up to six components, and provides an explicit solution to a question posed by Lin and addressed by Linell and Schick regarding the existence of non-abelian torsion-free quotients of the braid group. The proof relies on the diagrammatic theory of welded braids and uses the Artin representation. As a corollary, we obtain yet another proof that braid groups themselves are torsion-free.
\end{abstract}

\maketitle
\section*{Introduction}
Artin in \cite{ArtinBraid} is the first author to mention the notion of \emph{link-homotopy}, in the context of braids.\footnote{For links, this theory was deeply investigated by \cite{MilnorLinkgrp,Levine4comp,HabeggerLinHomotopy}.} This is an equivalence relation that allows continuous deformations during which two distinct components remain disjoint at all times, but each component can self-intersect. Subsequently, numerous authors have explored braids up to link-homotopy.

Among them, Goldsmith, in \cite{GoldsmithHomotopybraids}, answers a question of Artin about the distinction between isotopy and link-homotopy of braids, by providing an example of a non-trivial braid, up to isotopy, that is trivial up to link-homotopy. She also provided a presentation of the \emph{homotopy braid group}, which emerges as a quotient of the standard braid group. Habegger and Lin, in \cite{HabeggerLinHomotopy}, examined pure braids as reduced free group automorphisms, through a homotopy version of \emph{Artin's representation}. Additionally, Humphries in \cite{HumphriesTorsion} addressed the question of torsion in the homotopy braid group, proving torsion-freeness for six components or less.

This \emph{torsion problem} also arises in \cite{BardaVershinJieHomotopyBraid}, where the authors mention the broader question posed by Lin, formulated in \cite{LinBraidsPermutationPolynomials} and addressed in the Kourovka notebook \cite{MazurovKhukhroKourovkaNotebook}: \lq Is there a non-trivial epimorphism of the braid group onto a non-abelian group without torsion?\rq. Linnell and Schick, in \cite{LinnellchickFinititgrpextAtiyahconj}, provide a complete solution by showing that the braid group is residually torsion-free nilpotent-by-finite, hence in particular has plenty of non-trivial torsion-free quotients. However, they only give an existence proof, and explicit examples are not known for more than six components. Our main result is a complete solution to the torsion problem for any number of components.

\begin{thm0}(Theorem \ref{thmtorsionfree})
     The homotopy braid group is torsion-free for any number of components.
\end{thm0}

Additionally, as a corollary of this result, we recover the well-known result that the braid group is torsion-free for any number of components (Corollary \ref{coroBntorsionfree}). This fact, originally due to Fadell and Neuwirth in \cite[Thm. 8]{FadellNeuwirthConfSpace}, also follows from the stronger property of left-orderability established by Dehornoy in \cite{DehornoyBraidGrpLeftDistribuOp}. The property of left-orderability for the homotopy braid group is not known to this day and constitutes an interesting open question, as discussed in Remark \ref{rmqdehornoy}.

Our result relies heavily on the theory of \emph{welded braids}, which are a diagrammatic generalization of braids, allowing virtual crossings in addition to the classical ones. They are regarded up to certain local deformations that generalize the usual Reidemeister moves. Welded braids can be defined through various equivalent definitions that have been investigated by numerous authors across different contexts \cite{FennRimanyiRourkeBraidPermuGrp,SavushkinaGrpConjAutoFreeGrp,McCoolBasisConjAutoFreeGrp,BaezWiseCransExoticStatsStrings4DTheory,BrendleHatcherConfigspaceRingsWickets,ABMWhomotopyribbontube}. While we won't discuss these various perspectives here, we direct interested readers to Damiani’s survey \cite{celestejourneyloopbraidgrp} for further details. Instead, our focus lies on the notion of link-homotopy for welded objects, alongside its interpretation in terms of \emph{arrow calculus}.

The notion of \emph{$w$-arrows}, and more generally, \emph{$w$-trees}, developed in \cite{JBYasuArrowcalc}, is a welded version of Habiro's claspers \cite{HabiroClasp}. These are diagrammatic tools upon which surgery operations can be performed. Their manipulation up to link-homotopy is described in \cite[\S9]{JBYasuArrowcalc} in what is referred to as homotopy arrow calculus. This homotopy arrow calculus, will be central to our study.
\bigskip

The paper is divided into two sections. In the first preparatory section, we review the notion of homotopy welded braids along with its associated arrow calculus. Additionally, we recall the definition of the Artin representation in this context. Then, in the second section, we prove our main theorem in two steps: first, a topological characterization of torsion elements in the homotopy braid group using the Artin representation; then, an algebraic step where we show that the previously established characterization is never satisfied for classical braids.

\medskip\emph{Acknowledgement}: The author thanks Jean--Baptiste Meilhan and Akira Yasuhara for their careful reading and pertinent corrections. The author also thanks Luis Paris for an insightful comment that led to one of the key ideas.

\section{Preliminaries}
In this section, we review the notion of braids and their welded extension, focusing specifically on their study up to link-homotopy. Additionally, we present the theory of homotopy arrow calculus. Finally we recall the Artin representation in this context. 

\subsection{Welded Braids}\label{sectionwelded}

Let us take $n$ fixed points, in the unit interval $[0,1]$, denoted by $p_1<p_2<\cdots<p_n$.

\begin{defn}\label{defweldedbraid}
An \emph{$n$-component welded braid diagram} $\beta=(\beta_1,\ \ldots,\ \beta_n)$ is the oriented\footnote{In this paper, the adopted convention for the orientation of welded braid diagrams is from top to bottom.} image of an immersion \[(\beta_1,\ \ldots,\ \beta_n) :\underset{i\leq n}{\bigsqcup}[0,1] \to [0,1]\times [0,1]\] 
such that, for some permutation of $\{1,\ \ldots,\ n\}$ associated to $\beta$ and denoted by $\pi(\beta)$, we have $\beta_i(0)=(p_i,\ 0)$ and $\beta_i(1)=(p_{\pi(\beta)(i)},\ 1)$ for each $i$. We require the singularities to be a finite number of transverse double points, labeled either as classical crossings or as virtual crossings, as illustrated in Figure \ref{figcrossings}. Additionally, we require the immersion to be monotonic, which means that $\beta_i(t)\in [0,1]\times \{t\}$ for any $t\in[0,\ 1]$ and any $i$. We call the image of $\beta_i$, the \emph{$i$-th component} of $\beta$. 

\begin{figure}[!htbp]
    \centering
    \includegraphics{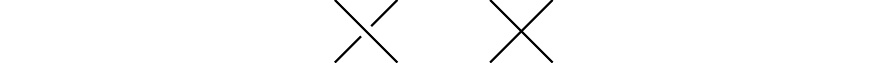}
    \caption{A classical and a virtual crossing.}
    \label{figcrossings}
\end{figure}
\end{defn}

A diagram whose associated permutation is the identity is said to be \emph{pure}. A diagram with no virtual crossings is called \emph{classical}.

\begin{defn}
\emph{Welded link-homotopy} is the equivalence relation generated by planar isotopies and the following local moves:
\begin{itemize}
    \item classical Reidemeister moves,
    \begin{figure}[!htbp]
    \centering
    \includegraphics[page=1]{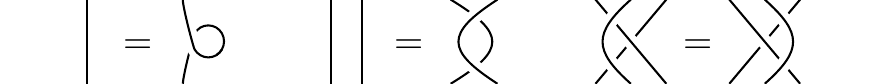}
    \end{figure}
    \item virtual Reidemeister moves,
    \begin{figure}[!htbp]
    \centering
    \includegraphics[page=2]{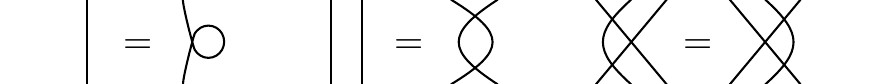}
    \end{figure}
    \item mixed Reidemeister move, $\hspace{4cm}\bullet$ OC move,
    \begin{figure}[!htbp]
    \centering
    \includegraphics[page=3]{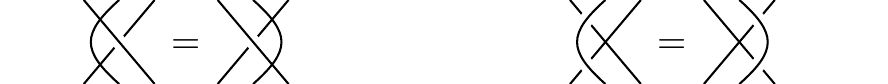}
    \end{figure}
    \item self-virtualization move.
    \begin{figure}[!htbp]
    \centering
    \includegraphics[page=4]{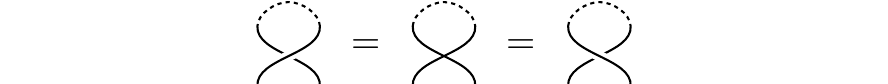}
    \end{figure}
\end{itemize}
\end{defn}

Generally, welded braid diagrams are studied up to \emph{welded isotopy}, which is defined by the same set of moves except for the self-virtualization move. 

Note that we do not require welded link-homotopy to preserve the monotonic property during the deformation. Indeed, in order to apply the self-virtualization move, we need self-crossings, which are prohibited by monotonicity. Furthermore, if a diagram is not monotonic, we can apply a welded link-homotopy to convert it into a monotonic diagram (see \cite[Thm. 4.1]{ABMWUsualVirtualWeldedObjHomotopy}). Therefore, since we are working up to link-homotopy, we can freely forget the monotonic condition, which will be useful later, as \emph{arrow surgery} typically does not preserve this property. 

Moreover, note that the self-virtualization move generates the \emph{self-crossing change}, where a classical crossing involving two strands from the same component is changed into its opposite. This proves, in particular, that the usual \emph{link-homotopy}\footnote{Link-homotopy, in classical braid theory, is the equivalence relation defined by planar isotopy, classical Reidemeister moves, and self-crossing changes.} implies the welded link-homotopy.

The set of welded braid diagrams up to welded link-homotopy, equipped with the stacking operation, forms a group: the \emph{homotopy welded braid group}, denoted by $hWB_n$. Elements of $hWB_n$ are called \emph{homotopy welded braids}. The set of pure welded braid diagrams up to welded link-homotopy forms a subgroup of $hWB_n$ denoted by $hWP_n$ and called the \emph{pure homotopy welded braid group}. Similarly, the set of classical braid diagrams up to link-homotopy, equipped with the stacking operation, forms a group called the \emph{homotopy braid group} and denoted by $hB_n$. Elements of $hB_n$ are called \emph{homotopy braids}. Finally, The set of pure classical braid diagrams up to link-homotopy forms a subgroup of $hB_n$ denoted by $hP_n$ and called the \emph{pure homotopy braid group}. 

In the following proposition, we recall the know fact that the set-theoretic inclusion of $hB_n$ in $hWB_n$ is injective; a proof can be found in \cite[Prop. 4.2.23]{Moithese}.

\begin{prop}\label{propbraidembedinwelded}
The homotopy braid group $hB_n$ injects into the homotopy welded braid group $hWB_n$.
\end{prop}

\subsection{Arrow calculus}
Let us now review the \emph{arrow calculus} developed by Meilhan and Yasuhara in \cite{JBYasuArrowcalc}. More precisely, let us consider the \emph{homotopy arrow calculus}, which deals with welded link-homotopy and will prove to be a central tool in our study.

\begin{defn}\cite[Def. 3.1]{JBYasuArrowcalc}
A \emph{$w$–tree} for a welded braid diagram $\beta$ is an immersed connected uni-trivalent tree $T$, such that
\begin{itemize}
    \item The univalent vertices of $T$ are pairwise disjoint and are contained in $\beta\setminus\{$crossings of $\beta\}$.
    \item There are finitely many singularities that are transverse double points of only two possible kinds:
    \begin{itemize}
        \item virtual crossings between edges of $T$,
        \item virtual crossings between strands of $\beta$ and edges of $T$.
    \end{itemize}
    \item All edges of $T$ are oriented, such that each trivalent vertex has two ingoing edges and one outgoing edge.
    \item Each edge of $T$ is assigned a number (possibly zero) of involutive\footnote{By involutive, we mean that twists on a same edge cancel pairwise.} decorations, called \emph{twists}, which are disjoint from all vertices and crossings.
\end{itemize}
A $w$–tree with a single edge is called a \emph{$w$–arrow}. We define the \emph{degree} of $T$, denoted by $\deg(T)$, as its number of trivalent vertices plus one.
\end{defn}

\vspace{-4pt}
\begin{figure}[!htbp]
    \centering
    \includegraphics{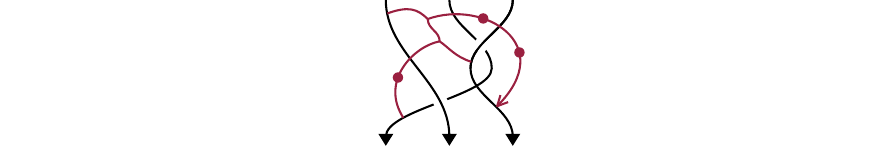}
    \vspace{-2pt}
    \caption{Example of a degree $3$ $w$-tree for a welded braid diagram.}
    \label{Weldedbraidarrowex}
\end{figure}

The unique univalent vertex with an ingoing edge is called the \emph{head} of the $w$-tree. By graphic convention, it is represented by an arrow in the figures. The other univalent vertices are called \emph{tails}. When we do not need to distinguish between tails and head, we simply call all univalent vertices, \emph{endpoints}. In the figures, portions of the diagram are represented by black lines and $w$-trees edges by red lines. Finally, twists are represented graphically by red dots \textcolor{arrow}{$\bullet$}. See Figure \ref{Weldedbraidarrowex} for an example.

Given a union of $w$-trees $F$ for a welded braid diagram $\beta$, there is a procedure called \emph{surgery} detailed in \cite{JBYasuArrowcalc} to construct a new diagram denoted $\beta_F$. We illustrate in Figure \ref{arrowsurgery} the surgery along a $w$-arrow.
\begin{figure}[!htbp]
    \centering
    \includegraphics[page=1]{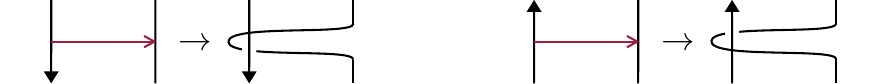}
    \caption{Surgery on a $w$-arrow.}
    \label{arrowsurgery}
\end{figure}
Note that the orientation of the strand containing the tail, needs to be specified to define the surgery move. In the case where a $w$-arrow contains some twist, surgery introduces a virtual crossing, as shown on the left-hand side of Figure \ref{arrowsurgery2}. Moreover, if the edge of the $w$-arrow intersects the diagram $\beta$, or another edge of a $w$-arrow, then the surgery introduces virtual crossings as indicated on the right-hand side of Figure \ref{arrowsurgery2}.

\begin{figure}[!htbp]
    \centering
    \includegraphics[page=2]{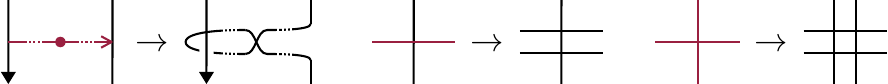}
    \caption{Surgeries near a twist and crossings.}
    \label{arrowsurgery2}
\end{figure}

Now if $F$ contains some $w$-trees with degree higher than one, we first apply the \emph{expanding rule} shown in Figure \ref{expandingrule}\footnote{Here and in the following figures, we use the diagrammatic convention adopted in \cite[Convention 5.1]{JBYasuArrowcalc}.} at each trivalent vertex: this breaks up $F$ into a union of $w$-arrows, on which we can perform surgery.

\begin{figure}[!htbp]
    \centering
    \includegraphics{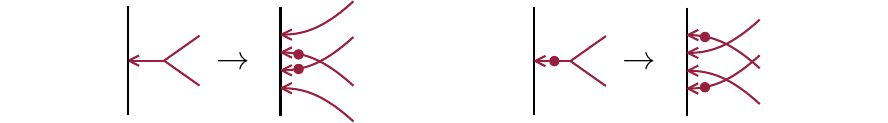}
    \caption{The expanding rule.}
    \label{expandingrule}
\end{figure}

Welded braid diagram composition yields a notion of \emph{product} for $w$-trees, as follows. We say that a union $F$ of $k$ $w$-trees ($k\geq2$) for the trivial braid $\mathbf 1$, is a product if \[\mathbf 1_F=\prod_{i=1}^k\mathbf 1_{T_i},\] where $T_i$ is a single $w$-tree for each $i$.

\emph{Homotopy arrow calculus} refers to the set of operations on unions of welded braid diagrams with some $w$-trees, which yield welded link-homotopic surgery results. These operations are defined and developed in \cite[\S9]{JBYasuArrowcalc}. In the rest of this section, we describe some of them, that will be used later.

\begin{defn}\label{defarrowiso}
\emph{Arrow isotopy} is the equivalence relation generated by planar isotopies, virtual Reidemeister moves involving edges of $w$-trees and/or strands of diagrams, and the following local moves:

\begin{figure}[!htbp]
    \centering
    \includegraphics{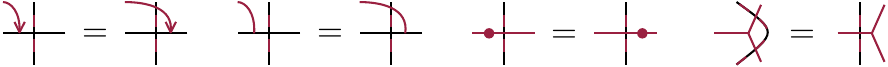}
\end{figure} 

\noindent here the vertical strands are either edges of $w$-trees or strands of diagrams.
\end{defn}
\begin{lem}\cite[Lem. 5.6]{JBYasuArrowcalc}\label{lemarrowmoves}
Two $w$-trees related by an arrow isotopy have welded isotopic, hence link-homotopic surgery results.
\end{lem}

\begin{defn}
A $w$-tree for a welded braid diagram $\beta$ is \emph{repeated} if it intersects a component of $\beta$ in at least two endpoints. 
\end{defn}
\begin{lem}\cite[Lem. 9.2 ]{JBYasuArrowcalc}\label{lemarrowrepeat}
Surgery along a repeated $w$-tree does not change the welded link-homotopy class of the welded braid diagram.
\end{lem}

The following lemma describes how to manipulate endpoints up to welded link-homotopy.
\begin{lem}\label{lemexchange}\cite{JBYasuArrowcalc} Let $T$ and $S$ denote two $w$-trees for a given welded braid diagram $\beta$. We have the following local moves up to welded link-homotopy.\footnote{By the notation $T=S$ we mean that $\beta_T$ and $\beta_S$ are equal in $hWB_n$}
\begin{itemize}
 \item Head/Tail reversal. If $T'$ is obtained from $T$ by modifying its head, resp. one of its tail, as shown on the left, resp. right of the figure, then $T=T'$. 
\end{itemize}
\begin{figure}[!htbp]
    \centering
    \includegraphics{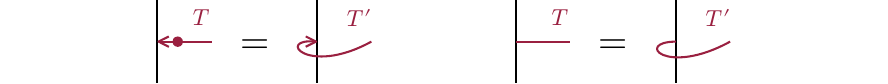}
\end{figure}
\begin{itemize}
\item \emph{Tails exchange}. If $T$ and $S$ have two adjacent tails and if $T'\cup S'$ is obtained from $T\cup S$ by exchanging these tails, then $T\cup S=T'\cup S'$ as shown on the left of the figure.
 \item \emph{Heads exchange}. If the heads of $T$ and $S$ are adjacent and if $T'\cup S'$ is obtained from $T\cup S$ by exchanging these heads as depicted in the middle of the figure, then $T\cup S=T'\cup S'\cup \tilde T$, where $\tilde T$ is as shown in the figure.
\item \emph{Head/Tail exchange}. If the head of $T$ is adjacent to a tail of $S$ and if $T'\cup S'$ is obtained from $T\cup S$ by exchanging these endpoints as depicted on the right of the figure, then $T\cup S=T'\cup S'\cup \tilde T$, where $\tilde T$ is as shown in the figure.
\end{itemize}

\begin{figure}[!htbp]
    \centering
    \includegraphics{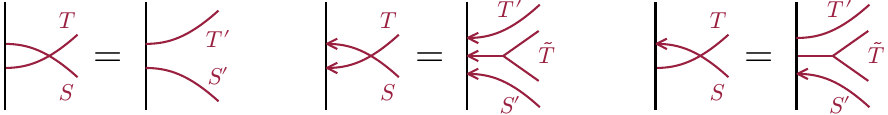}
\end{figure}
\end{lem}
We remark that the Head/Tail reversal, the Tails exchange and the Heads exchange moves are already valid up to welded isotopy.

\subsection{Artin representation}\label{sectionrpz}
Originally, the Artin representation was defined in \cite{ArtinBraid} within the framework of classical braids, up to isotopy. Subsequently, it has been extended and extensively studied in the contexts of link-homotopy and/or welded braids. Below, we focus on its welded version, up to link-homotopy.

Recall the Artin braid generators $\sigma_i$ for $i\in\{1,\ \ldots,\ n-1\}$ illustrated in Figure \ref{artingene} and the virtual braid generators $\rho_i$ for $i\in\{1,\ \ldots,\ n-1\}$, illustrated in Figure \ref{puregene}.
\begin{figure}[!htbp]
    \begin{minipage}[c]{.46\linewidth}
        \centering
    \includegraphics[page=1]{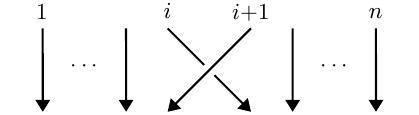}
        \caption{The Artin generator $\sigma_i$.}
        \label{artingene}
    \end{minipage}
    \hfill%
    \begin{minipage}[c]{.46\linewidth}
        \centering
    \includegraphics[page=2]{weldedbraidgene.pdf}
        \caption{The virtual braid generator $\rho_i$.}
        \label{puregene}
    \end{minipage}
\end{figure}

We also need the \emph{reduced free group} denoted $\mathcal{R}F_n$, given by $n$ generators $x_1$, $x_2$, $\ldots$, $x_n$, subject to the relations $[\omega x_i\omega^{-1},x_i]=1$ for any $1\leq i\leq n$ and any word $\omega$ in $x_1,\,\ldots, x_n$.

Let us now define the \emph{homotopy welded Artin representation}.
\begin{defn}\label{defweldedartinrpz}
We call \emph{homotopy welded Artin representation} the homomorphism denoted by $\phi :hWB_n\rightarrow Aut(\mathcal{R}F_n)$. It is defined for $i\in\{1,\,\ldots,\ n-1\}$ as follows:
 \[ \phi(\rho_i): \left \{ \begin{array}{llll}
    x_i&\mapsto&x_{i+1},&\\
    x_{i+1} &\mapsto&x_i,&\\
    x_k&\mapsto&x_k&\mbox{if }k\notin\{i,\ i+1\},
 \end{array}
 \right.\]
 and
\[ \phi(\sigma_i): \left \{ \begin{array}{llll}
    x_i&\mapsto&x_{i+1},&\\
    x_{i+1} &\mapsto&x_{i+1}^{-1}x_ix_{i+1},&\\
    x_k&\mapsto&x_k&\mbox{if }k\notin\{i,\ i+1\}.
 \end{array}
 \right.\]
\end{defn}
\begin{ex}\label{excalcartinpure}
For any $1\leq i\neq j\leq n$, consider the pure homotopy welded braid generators: 
 \begin{figure}[!htbp]
    \begin{minipage}[c]{.46\linewidth}
        \centering
    \includegraphics[page=1]{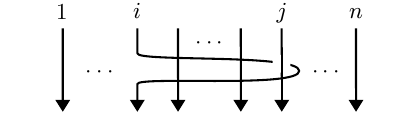}

    $\chi_{ij}=\rho_{i}\cdots\rho_{j-2}\sigma_{j-1}\rho_{j-1}\cdots\rho_{i}$ if $i<j$,
    \end{minipage}
    \hfill%
    \begin{minipage}[c]{.46\linewidth}
        \centering
    \includegraphics[page=2]{torsionchi_ij.pdf}

    $\chi_{ij}=\rho_{i-1}\cdots\rho_{j}\sigma_{j}\rho_{j+1}\cdots\rho_{i-1}$ if $j<i$.
    \end{minipage}
\end{figure}

A direct computation gives \[\Phi(\chi_{ij})(x_k)=\left\{\begin{array}{ll}
    x_k&\mbox{if }k\neq i,\\
    x_j^{-1}x_ix_j&\mbox{if }k=i.
 \end{array}\right.\]
This defines the \lq pure\rq\ part of the representation, which already appears in \cite{ABMWhomotopyribbontube}. There, the authors consider two equivalent definitions of the representation: a geometric one, in terms of ribbon tubes, and a combinatorial one using Gauss diagrams.
\end{ex}
    
\begin{prop}\label{propartininj}
The homotopy welded Artin representation $\phi$ is injective.
\end{prop}
\begin{proof}
Let us first note that for any $\beta\in hWB_n$ the homomorphism $\phi(\beta)$ sends a generator $x_i$ to a conjugates of $x_{\pi^{-1}(\beta)(i)}$. In particular, if $\beta\in\ker(\phi)$ then $\beta$ is a pure braid. Therefore, we conclude with \cite[Thm. 2.34]{ABMWhomotopyribbontube}, which states that the representation $\phi$ restricted to $hWP_n$ is injective.
\end{proof}

\begin{lem}\label{lemsondagewelded}
Let $F$ be a $w$-tree for the trivial braid $\mathbf{1}\in hWB_n$ not having its head on the $k$-th component. The action of $\phi(\mathbf1_F)$ on $x_k \in \mathcal{R}F_n$, is given by: $$\phi(\mathbf1_{F})(x_k)=x_k.$$
\end{lem}
\begin{proof}
By using the expanding rule of Figure \ref{expandingrule} iteratively, we can break up $F$ into a union of $w$-arrows, none of which has its head on the $k$-th component. Then, using arrow isotopies, we can rearrange such a union into a product. It is therefore enough to prove the result when $F$ is of degree one. In this case, up to arrow isotopy and the Head/Tail reversal move of Lemma \ref{lemexchange}, we can assume that $F$ is simply a horizontal $w$-arrow. Then, $\mathbf1_F$ is either $\chi_{ij}$ or $\chi_{ij}^{-1}$ for some $1 \leq i \neq j \leq n$ with $k \neq i$, and the result follows from Example \ref{excalcartinpure}.
\end{proof}

\section{Proof of Torsion-Freeness}\label{sectiontorsion}

This section contains the main result of our paper: the proof that homotopy braid groups are torsion-free (Theorem \ref{thmtorsionfree}). The proof is decomposed into two steps. In the first step, we show that torsion elements in $hWB_n$ appear as conjugates of specific braids, having only virtual crossings. Then, in a second step, using an algebraic obstruction, we show that such conjugates never correspond to classical braids in $hB_n$.

\subsection{A characterization of torsion}
Let us denote by $\lambda_n\in hWB_n$ the homotopy welded braid, illustrated in Figure \ref{figlambdawelded}, given by
\[\lambda_n=\rho_1\rho_2\cdots\rho_{n-1}.\]

\begin{figure}[!htbp]
    \centering
    \includegraphics{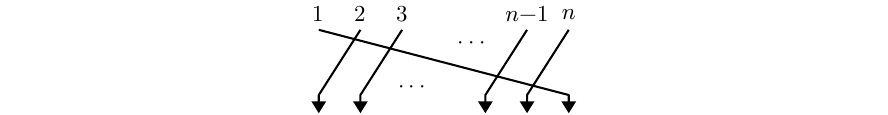}
    \caption{The homotopy welded braid $\lambda_n$.}
    \label{figlambdawelded}
\end{figure}

Notice that $\lambda_n$ is a torsion element of order $n$. We denote by $\tau_n$ the $n$-cycle $(n\ n-1\ \cdots\ 2\ 1)=\pi(\lambda_n)$ associated to $\lambda_n$. When the value of $n$ is clear from the context, it will be omitted in the notation. Let us describe in the following lemma the action of conjugation by $\lambda$ on $w$-trees.

\begin{lem}\label{lemweldedtreeconj}
Let $F$ be a product of $w$-trees for the trivial braid $\mathbf{1} \in hWB_n$, all having their heads on the $k$-th component. Then, the conjugate $\lambda \mathbf{1}_F \lambda^{-1}$ is link-homotopic to $\mathbf{1}_{F'}$, where $F'$ is a product of $w$-trees, all having their heads on the component $\tau^{-1}(k).$
\end{lem}
\begin{proof}
Since the conjugate of a product is the product of the conjugates, it is enough to prove the result when $F$ is a single $w$-tree. Consider, therefore, the conjugate $\lambda\mathbf{1}_F\lambda^{-1}$ where $F$ is a single $w$-tree having its head on the $k$-th component. We use an arrow isotopy to slide $F$ through $\lambda$. This operation turns $F$ into a new $w$-tree whose head is on component $\tau^{-1}(k)$. Finally, simplify $\lambda$ with $\lambda^{-1}$ by applying a welded isotopy.
\end{proof}

In the following lemma, we use arrow calculus to describe specific conjugates of homotopy welded braids. Later on Lemma \ref{lemtorsionweldedconj}, we will use this lemma to characterize torsion elements in $hB_n$ up to conjugation.

\begin{lem}\label{lemtorsionweldedniceform}
Let $\beta\in hWB_n$ be a homotopy welded braid, whose associated permutation is an $n$-cycle. Then $\beta$ is conjugate to $\mathbf 1_{F}\cdot\lambda$, where $F$ is a product of $w$-trees all having their heads on the $n$-th component.
\end{lem}
\begin{proof}
Let us show by induction on $k\in\{1,\ \ldots,\ n\}$ that $\beta$ is conjugate to a braid of the form $\mathbf1_{F_k}\cdot\lambda$, where $F_k$ is a product of $w$-trees having their heads on components numbered from $k$ to $n$.

For the base case, we can assume that $\beta$ is conjugated with $\beta_1$ satisfying $\pi(\beta_1)=\tau=(n\ n-1\ \cdots\ 2\ 1)$. We can then write $$\beta_1=\theta\lambda,$$ where $\theta=\beta_1\lambda^{-1}$ is a pure homotopy welded braid. Additionally, it is clear that a virtualization move can be achieved by surgery on a $w$-arrow (see \cite[Fig. 6]{JBYasuArrowcalc}). This implies that any pure braid can be obtained by $w$-arrow surgeries on the trivial braid. We can thus write
$$\beta_1=\mathbf1_{F_1}\cdot\lambda,$$
where $F_1$ is a product of $w$-arrows.

Now, assume by induction that for some $1\leq k<n$ the braid $\beta$ is conjugate with $$\beta_k=\mathbf1_{F_k}\cdot\lambda,$$
where $F_k$ is a product of $w$-trees with their heads on components numbered from $k$ to $n$.
Note, on one hand, that Lemma \ref{lemexchange} allows us to exchange the relative position of two consecutive factors in a product, up to higher degree $w$-trees. On the other hand, note that by Lemma \ref{lemarrowrepeat}, $w$-trees of degree higher than $n$ are repeated and are therefore trivial up to link-homotopy. Combining these two facts,\footnote{Note that a similar sorting process appears in the proof of \cite[Thm. 9.4]{JBYasuArrowcalc}; see also \cite[Thm. 4.3]{YasuharaAkiraSelfdelteq} for an analogous argument using claspers.} we rearrange $w$-trees degree by degree to obtain the following claim.
\begin{claim}\label{claimarrowdecompo}
The pure homotopy welded braid $\mathbf1_{F_k}$ decomposes as follows:
$$\mathbf1_{F_k}=\mathbf1_{F'_k}\mathbf1_{F''_k},$$ where $F'_k$ is a product of $w$-trees having their heads on the $k$-th component and $F_k''$ is a product of $w$-trees having their heads on components numbered from $k+1$ to $n$.
\end{claim}

We then consider and compute the conjugate,
\begin{align*}
    \beta_{k+1}:=&\,(\mathbf1_{F'_k})^{-1}\beta_k\mathbf1_{F'_k},\\
    =&\,(\mathbf1_{F'_k})^{-1}\mathbf1_{F'_k}\mathbf1_{F''_k}\cdot\lambda\mathbf1_{F'_k}\lambda^{-1}\cdot\lambda,\\
    =&\,\mathbf1_{F''_k}\cdot\lambda\mathbf1_{F'_k}\lambda^{-1}\cdot\lambda.
\end{align*}
Finally, by Lemma \ref{lemweldedtreeconj}, $\lambda\mathbf1_{F'_k}\lambda^{-1}$ is a product of $w$-trees having their heads on the $(k+1)$-th component which concludes the induction as well as the proof.
\end{proof}

Let us now tackle the question of torsion, starting with a preliminary result for pure braids. This proposition is well-known and can be found, for example, in \cite{HabeggerLinHomotopy} and \cite{HumphriesTorsion}.

\begin{prop}\label{prophPntorsionfree}
The pure homotopy braid group $hP_n$ is torsion-free for any positive integer $n$.
\end{prop}

With the following lemma, we reduce our study to prime numbers of components.

\begin{lem}\label{lemtorsionorder}
If there is torsion in $hB_n$, then for some prime number $p \leq n$, there exists a torsion element of order $p$ in $hB_p$. 
\end{lem}
\begin{proof}
Let $\beta \in hB_n$ be a torsion element of prime order $p$, and let $\pi(\beta)$ be its associated permutation. By Proposition \ref{prophPntorsionfree}, the subgroup $hP_n$ is torsion-free. Consequently, $\pi(\beta) \neq 1$, implying that $\pi(\beta)$ is a torsion element of order $p$ in the symmetric group $S_n$. Specifically, $\pi(\beta)$ can be expressed as a product of disjoint $p$-cycles, where $p \leq n$. Let us denote one such $p$-cycle as $(i_1, \ldots, i_p)$. Now, consider the subgroup $G$ of $hB_n$ generated by elements whose associated permutation sends the set $\{i_1, \ldots, i_p\}$ to itself. Next, define the homomorphism from $G$ to $hB_p$, which retains only the components $\{i_1, \ldots, i_p\}$. This homomorphism sends $\beta$ to a torsion element of order $p$ in $hB_p$, thereby completing the proof.
\end{proof}

Recall from Definition \ref{defweldedbraid} that a homotopy welded braid is \emph{classical} if it belongs to the subgroup $hB_n$ of $hWB_n$ generated by the Artin generators $\sigma_i$ (see Proposition \ref{propbraidembedinwelded}). Note also that the welded braid $\lambda\in hWB_n$ is a torsion element of order $n$. In the following lemma, we give a characterization of torsion elements in $hB_n$ using this braid $\lambda$ as well as the notion of classical braid.
\begin{lem}\label{lemtorsionweldedconj}
There is torsion in $hB_n$ if and only if for some prime number $p\leq n$ the braid $\lambda\in hWB_p$ given by $\lambda=\rho_1\rho_2\cdots\rho_{p-1}$ is conjugate to a classical braid.
\end{lem}
\begin{proof}
According to Lemma \ref{lemtorsionorder}, if there is torsion in $hB_n$, we can find a torsion element $\beta$ of order $p$ in $hB_p$ for some prime number $p$. Consequently, $\pi(\beta)^p=\mathrm{Id}$, and by Proposition \ref{prophPntorsionfree}, $\pi(\beta)\neq\mathrm{Id}$. Thus, $\pi(\beta)$ is a torsion element of order $p$ and therefore it is a $p$-cycle. Now, by Proposition \ref{propbraidembedinwelded}, we can consider the classical braid $\beta$ as an element of $hWB_p$. Therefore, by Lemma \ref{lemtorsionweldedniceform}, $\beta$ is conjugate to $\mathbf 1_{F}\cdot\lambda$, where $F$ is a product of $w$-trees, all having their heads on the $p$-th component. Moreover, by Lemma \ref{lemweldedtreeconj}, for any integer $k\in\{1,\ \ldots,\ p-1\}$, the conjugates $\lambda^k\mathbf 1_{F}\lambda^{-k}$ are products of $w$-trees, none of which have their heads on the $p$-th component. Hence, by Lemma \ref{lemsondagewelded},$$\phi\left(\lambda^k\mathbf 1_{F}\lambda^{-k}\right)(x_p)=x_p$$ for any $k\in\{1,\ \ldots,\ p-1\}$. In particular:
\begin{align*}
    \phi\big((\mathbf 1_{F}  \lambda)^p\big)(x_p)&=\phi\Big(\mathbf 1_{F}\big(\lambda\mathbf 1_{F}\lambda^{-1}\big)\big(\lambda^2\mathbf 1_{F}\lambda^{-2}\big)\cdots\big(\lambda^{p-1}\mathbf 1_{F}\lambda^{1-p}\big)\lambda^p\Big)(x_p),\\
    &=\phi(\mathbf 1_{F})\circ\phi\big(\lambda\mathbf 1_{F}\lambda^{-1}\big)\circ\phi\big(\lambda^2\mathbf 1_{F}\lambda^{-2}\big)\circ\cdots\circ\phi\big(\lambda^{p-1}\mathbf 1_{F}\lambda^{1-p}\big)(x_p),\\
    &=\phi(\mathbf 1_{F})(x_p).
\end{align*}
On the other hand, since $\beta$ is a torsion element, $\beta^p=(\mathbf 1_{F}\lambda)^p=\mathbf1$, which implies that $$\phi\big((\mathbf 1_{F}\lambda)^p\big)(x_p)=\phi(\mathbf 1)(x_p)=x_p.$$
By combining the two previous equalities, we deduce that $\phi(\mathbf 1_{F})(x_p)=x_p$. Moreover, by Lemma \ref{lemsondagewelded} again, we also have that $\phi(\mathbf 1_{F})(x_k)=x_k$ for any $k<p$. Thus, $\phi(\mathbf 1_{F})=\mathrm{Id}$ and by injectivity of $\phi$ (Proposition \ref{propartininj}), the braid $\mathbf 1_{F}$ is trivial. Therefore, the classical braid $\beta$ is conjugate to $\lambda$, thus showing the \lq only if\rq\, part of the statement.

To show the converse implication, we use the fact that any conjugate of $\lambda$ is a torsion element of order $p$ in $hWB_p$ and that, consequently, a braid given by the same expression in $hB_n$ is also a torsion element.
\end{proof}

\subsection{An algebraic obstruction}
By Lemma \ref{lemtorsionweldedconj}, the final part of the proof of our main result consists in showing that for any integer $n$, the braid $\lambda$ has no classical braid as conjugate. To achieve this, we need to determine whether a welded braid is also a classical one. This is done in the following lemma, drawing inspiration from \cite[Thm. 1.7]{HabeggerLinHomotopy}.
\begin{lem}\label{lemclassicalcara}
Let $\beta\in hWB_n$ be a homotopy welded braid. If $\beta$ is a classical braid then $$\phi(\beta)(x_1x_2\cdots x_n)=x_1x_2\cdots x_n.$$
\end{lem}
\begin{proof}
This follows from a simple verification on the classical generators $\sigma_i$.
\end{proof}

Recall now the \emph{reduced Magnus expansion}. This is the homomorphism \[M: \mathcal RF_n\to \mathcal A_n\] from the reduced free group into the polynomial algebra $\mathcal A_n$ in non-commuting variables $X_1,\ \ldots,\ X_n$ in which monomials $X_{\alpha_1}X_{\alpha_2}\cdots X_{\alpha_k}$ vanish if $\alpha_i=\alpha_j$ for some $i\neq j$. The image of a generator $x_i$ is defined by the polynomial $M(x_i):=1+X_i$. 

In the following lemma, we describe elements not belonging to the fixed points of $\phi(\lambda)$.\begin{lem}\label{lemnofixptswelded}
For any $\beta\in hWB_n$, $\phi(\lambda)\circ\phi(\beta)(x_1x_2\cdots x_n)\neq\phi(\beta)(x_1x_2\cdots x_n)$.
\end{lem}
\begin{proof}
Let us define the map $F:\mathcal{A}_n\to\mathbf{Z}$ which sends a polynomial to the sum of the coefficients of its monomials of degree $n$. On the one hand, we observe that 
\[F\Big(M\big(\phi(\rho_i)(x)\big)\Big)=F\big(M(x)\big) \quad \text{and}\quad F\Big(M\big(\phi(\sigma_i)(x)\big)\Big)=F\big(M(x)\big),\] for any $x\in\mathcal RF_n$ and any $i$. The first equality is clear since $M\big(\phi(\rho_i)(x)\big)$ is obtained from $M(x)$ by permuting the variables $X_i$ and $X_{i+1}$. The second equality is less obvious: indeed, $M\big(\phi(\sigma_i)(x)\big)$ is obtained from $M(x)$ by first permuting the variables $X_i$ and $X_{i+1}$ and then substituting $X_i$ with $X_i + X_iX_{i+1} - X_{i+1}X_i$, potentially introducing new monomials of degree $n$. However, these extra monomials appear in pairs and with opposite signs and thus do not change the value of $F$. Therefore, we have that 
$$F\big(M\big(\phi(\beta)(x)\big)\big)=F\big(M(x)\big),$$
for any $\beta\in hWB_n$ and any $x\in\mathcal RF_n$. Moreover note that,
\begin{align*}
    F\big(M(x_1x_2\cdots x_n)\big)&=F\big((1+X_1)(1+X_2)\cdots(1+X_n)\big),\\
    &=1,
\end{align*}
hence $F\big(M\big(\phi(\beta)(x_1x_2\cdots x_n)\big)\big)=1$ for any $\beta\in hWB_n$.
But on the other hand, for any $\omega\in\mathcal RF_n$, $M\big(\phi(\lambda)(\omega)\big)$ is obtained from $M(\omega)$ by permuting the variables $X_i$ cyclically. Hence a necessary condition for having $M\big(\phi(\lambda)(\omega)\big)=M(\omega)$ is that $$F\big(M(\omega)\big)\equiv0\quad(\mathrm{mod}\ n).$$
Therefore, $\phi(\lambda)\circ\phi(\beta)(x_1x_2\cdots x_n)\neq\phi(\beta)(x_1x_2\cdots x_n)$ for any $\beta\in hWB_n$.
\end{proof}

We can finally prove the main theorem of this paper.
\begin{thm}\label{thmtorsionfree}
The homotopy braid group $hB_n$ is torsion-free for any number of components $n$.
\end{thm}
\begin{proof}
Suppose by contradiction that there is a torsion element in $hB_n$. By Lemma \ref{lemtorsionweldedconj} there exists a prime number $p\leq n$ and some braid $\beta\in hWB_p$ such that $\beta^{-1}\lambda\beta$ is a classical braid. According to Lemma \ref{lemclassicalcara} this conjugate must satisfy,
$$\phi\big(\beta^{-1}\lambda\beta\big)\big(x_1x_2\cdots x_p)=x_1x_2\cdots x_p,$$
or equivalently,
$$\phi(\lambda)\circ \phi(\beta)(x_1x_2\cdots x_p)= \phi(\beta)(x_1x_2\cdots x_p).$$
This yields a contradiction by Lemma \ref{lemnofixptswelded}.
\end{proof}

It follows from Theorem \ref{thmtorsionfree} that the standard braid group $B_n$ is torsion-free for all $n$. To prove this corollary, we need the following well-known proposition, which essentially dates back to Artin \cite{ArtinBraid}.
\begin{prop}\label{propPntorsionfree}
The pure braid group $P_n$ is torsion-free for any number of components $n$.
\end{prop}

We recover in this way a result of Fadell and Neuwirth (see Remark \ref{rmqdehornoy}).
\begin{cor}\label{coroBntorsionfree}
The braid group $B_n$ is torsion-free for any number of components $n$.
\end{cor}
\begin{proof}
Let us consider the projection $p: B_n \to hB_n$. Since $hB_n$ is torsion-free (Theorem \ref{thmtorsionfree}), any torsion element in $B_n$ must belong to the kernel $K := \ker(p)$. However, it is clear from \cite{GoldsmithHomotopybraids} that $K \subset P_n$, thus $K$ is torsion-free by Proposition \ref{propPntorsionfree} and the proof is complete.
\end{proof}
\begin{rmq}\label{rmqdehornoy}
The study of torsion in braid groups dates back to Fadell and Neuwirth in 1962.\footnote{We mention however a more geometric proof, given in \cite[\S2]{GonzalesBasicResultBraidGrp} and based on the earlier works of \cite{NielsenAbbildungsklassenEndlicherOrdnung,KerekjartoUberPerioTransfoKreisscheibeKugelflache,EilenbergTransfoPerioSurfSphere}.} Building upon topological methods, they show in \cite[Thm. 8]{FadellNeuwirthConfSpace} that $B_n$ is torsion-free for every $n$. One may wonder whether a similar methods may provide alternative proof of the torsion-freeness in homotopy braid groups. However, to date, there is no known classifying space for $hB_n$ that would enable such a demonstration. Moreover, another common proof of torsion-freeness in $B_n$ appears in \cite{DehornoyBraidGrpLeftDistribuOp}, where Dehornoy established the stronger property of left-orderability. One may wonder again if a similar proof applies for $hB_n$, but the question of orderability in this group is still open.
\end{rmq}

\bibliographystyle{alpha}
\bibliography{bibli.bib}
\addcontentsline{toc}{chapter}{Bibliography}

\end{document}